\documentclass[12pt]{amsart}
\usepackage{amsmath}
\usepackage{graphics, graphicx}
\usepackage[english,  activeacute]{babel}
\usepackage[latin1]{inputenc}
\usepackage{amssymb}
\usepackage{amsthm}
\usepackage{MnSymbol}
\usepackage{array}
\usepackage{pdflscape}
\usepackage{a4wide}
\usepackage{color, url}
\usepackage{float}
\usepackage{pstricks-add}

\theoremstyle{plain}
\newtheorem{theorem}{Theorem}
\newtheorem{corollary}[theorem]{Corollary}
\newtheorem{lemma}[theorem]{Lemma}

\theoremstyle{definition}

\newcommand{\W}{{\mathcal W}}

\newcommand{\N}{{\mathbb N}}

\newcommand{\F}{{\mathcal F}}
\newcommand{\T}{{\mathcal T}}

\title{A $q$-analogue of the Biperiodic Fibonacci Sequence}

\author{Jos\'e L. Ram\'{\i}rez}
\address{\noindent Departamento  de Matem\'aticas, Universidad Sergio Arboleda, Bogot\'a,  COLOMBIA}
\email{josel.ramirez@ima.usergioarboleda.edu.co}
\urladdr{http://sites.google.com/site/ramirezrjl}

\author{V\'{\i}ctor F. Sirvent}
\address{\noindent Departamento de Matem\'aticas, Universidad Sim\'{o}n Bol\'{i}var, Apartado 89000, Caracas 1086-A, VENEZUELA.}
\email{vsirvent@usb.ve}
\urladdr{http://www.ma.usb.ve/\textasciitilde  vsirvent}

\date{\today}
\subjclass[2010]{Primary  ; Secondary .}

\begin{document}
\begin{abstract}
The Fibonacci sequence has been generalized in many ways. One of them is defined by the relation $t_n=at_{n-1}+t_{n-2}$ if $n$ is even, $t_n=bt_{n-1}+t_{n-2}$ if $n$ is odd, with initial values  $t_0=0$ and  $t_1=1$, where $a$ and $b$ are positive integers. This sequence is called biperiodic Fibonacci sequence. In this paper, we introduce a $q$-analogue of this sequence. We prove several identities of $q$-analogues of the Fibonacci sequence.  We give algebraic and combinatorial  proofs.
\end{abstract}

\maketitle

\section{Introduction}

The Fibonacci numbers $F_n$ are defined  by the recurrence relation
\begin{align*}
F_{0}=0, \quad F_{1}=1,  \quad F_{n+1}=F_{n}+F_{n-1},  \ \ n\geqslant 1.
\end{align*}
This sequence and their generalizations have many interesting combinatorial properties (cf.  \cite{koshy}).  Many kinds of generalizations of Fibonacci numbers have been  presented in the literature. In particular, Edson  and Yayenie introduced the bi-periodic Fibonacci sequence \cite{edson1}.  For any two positive integers  $a$ and $b$, the bi-periodic Fibonacci sequence, say $\{t_{n}\}_{n\in \N}$, is determined by:
\begin{align}
t_{0}=0, \  \ t_{1}=1,  \   \  t_{n}=\begin{cases} at_{n-1} + t_{n-2},  \ \text{if} \ n\equiv 0 \pmod 2;\\
bt_{n-1} + t_{n-2},  \ \text{if} \  n\equiv 1\pmod 2;
\end{cases} \ n\geqslant 2. \label{eq1}
\end{align}
 These numbers and their generalizations have been studied in several papers; among other references, see  \cite{ALP, Bil, edson3,  edson1, NURE,   Panario, Panario2, Ram, edson2}. Yayenie \cite{edson2} found the following explicit formula to  bi-periodic Fibonacci numbers
  \begin{align}
t_{n}=a^{\xi(n-1)}\sum_{i=0}^{\left\lfloor\frac{n-1}{2}\right\rfloor}\binom{n-i-1}{i}(ab)^{\left\lfloor\frac{n-1}{2}\right\rfloor-i},  \label{ec20}
\end{align}
where $\xi(n)=n-2\lfloor  n/2 \rfloor$, i.e., $\xi(n)=0$ when $n$ is even and $\xi(n)=1$ when $n$ is odd. \\

There exists several slightly different $q$-analogues of the Fibonacci sequence, among other references, see, \cite{Andrews2, Bel, Bel2, Carlitz, Carlitz2, Cigler, Cigler2, Cigler3, Cigler4, Sagan}. In particular, Schur \cite{Schur} defined the following polynomials

\begin{align}
D_0(q)=0, \quad D_1(q)=1, \quad  D_n(q)=D_{n-1}(q)+q^{n-2}D_{n-2}(q). \label{shurp}
\end{align}

It is clear that $D_n(1)=F_n$.  Besides the recurrence formula (\ref{shurp}), $D_n(q)$ can be calculated directly by the following analytic formula (cf.  \cite{Andrews2, Cigler})
\begin{align}
D_n(q)=\sum_{j=0}^{\left\lfloor\frac{n-1}{2}\right\rfloor}{n - j -1 \brack j}q^{j^2},
\end{align}
  where the $q$-binomial is  $${n \brack k}:=\frac{(q;q)_n}{(q;q)_k(q;q_{n-k})},$$ and
$$(a;q)_n:=\prod_{j=0}^{n-1}(1-aq^j).$$
Another way to write the $q$-binomial is
 $$ {n \brack k}= \frac{ [n]_q!}{ [k]_q! [n-k]_q!},$$ with
$[n]_q=1+q+\cdots  + q^{n-1}$ and  $[n]_q!=[1]_q[2]_q\cdots [n]_q$.\\

One of the main applications of the polynomial as  $D_n(q)$ is the giving alternatives proofs of the  Rogers-Ramanujan identities; among other references, see \cite{Andrews4, Andrews2, Andrews3,  Andrews, Chan, Prodinger}.

A natural question is: What is the $q$-analogue of the biperiodic sequence?  From Definition (\ref{shurp}), we introduce a $q$-analogue of the biperiodic Fibonacci sequence as follows

\begin{align}
F_{0}(q)=0, \  \ F_{1}(q)=1,  \   \  F_{n}(q)=\begin{cases} aF_{n-1}(q) + q^{n-2}F_{n-2}(q),  \ \text{if} \ n\equiv 0 \pmod 2;\\
bF_{n-1}(q) + q^{n-2}F_{n-2}(q),  \ \text{if} \  n\equiv 1\pmod 2;
\end{cases} \ n\geqslant 2. \label{eqq1}
\end{align}

The first few terms are
$$0, \quad 1,  \quad a,   \quad a b + q,  \quad a^2 b + a q + a q^2,  \quad
a^2 b^2 + a b q + a b q^2 + a b q^3 + q^4, \dots$$

We call this sequence \emph{$q$-biperiodic Fibonacci sequence}.
It is clear that if $a=b=1$ we obtain the polynomials $D_n(q)$.

In the present article, we study the $q$-biperiodic Fibonacci sequence. We obtain new recurrence relations, new   combinatorial identities and the generating function of the $q$-biperiodic Fibonacci sequence. Finally, we introduce the tilings weighted by bicolored squares, then we give several combinatorial proof of some identities.

\section{Some Identities}

The following theorem gives a $q$-analogue of the identity (\ref{ec20}).  This is a close formula to evaluate the $q$-biperiodic Fibonacci sequence. Moreover, in Theorem \ref{theo4} we give a $q$-analogue of a generalization of Cassini's identity.

\begin{theorem}\label{Theo1}
The following equality holds for any integer $n\geqslant 0:$
\begin{align*}
F_n(q)=a^{\xi(n-1)}\sum_{l=0}^{\left\lfloor\frac{n-1}{2}\right\rfloor}{n-l-1 \brack l}(ab)^{\left\lfloor\frac{n-1}{2}\right\rfloor-l}q^{l^2}.
\end{align*}
\end{theorem}
\begin{proof}
We proceed by induction on $n$. The equality clearly holds for $n=0, 1$. Now suppose that the result is true for all $i\leqslant n$.

\begin{align*}
F_{n+1}(q)&=a^{\xi(n)}b^{1-\xi(n)}F_{n}(q)  + q^{n-1}F_{n-1}(q)\\
&=a^{\xi(n)}b^{1-\xi(n)}a^{\xi(n-1)}\sum_{l=0}^{\left\lfloor\frac{n-1}{2}\right\rfloor}{n-l-1 \brack l}(ab)^{\left\lfloor\frac{n-1}{2}\right\rfloor-l}q^{l^2} + q^{n-1}a^{\xi(n-2)}\sum_{l=0}^{\left\lfloor\frac{n-2}{2}\right\rfloor}{n-l-2 \brack l}(ab)^{\left\lfloor\frac{n-2}{2}\right\rfloor-l}q^{l^2}\\
&=ab^{\xi(n+1)}\sum_{l=0}^{\left\lfloor\frac{n-1}{2}\right\rfloor}{n-l-1 \brack l}(ab)^{\left\lfloor\frac{n-1}{2}\right\rfloor-l}q^{l^2}  + q^{n-1}a^{\xi(n)}\sum_{l=0}^{\left\lfloor\frac{n-2}{2}\right\rfloor}{n-l-2 \brack l}(ab)^{\left\lfloor\frac{n-2}{2}\right\rfloor-l}q^{l^2}\\
&=a^{\xi(n)}(ab)^{\xi(n+1)}\sum_{l=0}^{\left\lfloor\frac{n-1}{2}\right\rfloor}{n-l-1 \brack l}(ab)^{\left\lfloor\frac{n-1}{2}\right\rfloor-l}q^{l^2}  + q^{n-1}a^{\xi(n)}\sum_{l=0}^{\left\lfloor\frac{n-2}{2}\right\rfloor}{n-l-2 \brack l}(ab)^{\left\lfloor\frac{n}{2}\right\rfloor-l-1}q^{l^2}\\
&=a^{\xi(n)}\left(\sum_{l=0}^{\left\lfloor\frac{n-1}{2}\right\rfloor}{n-l-1 \brack l}(ab)^{\left\lfloor\frac{n-1}{2}\right\rfloor-l+\xi(n+1)}q^{l^2}  + q^{n-1}\sum_{l=1}^{\left\lfloor\frac{n}{2}\right\rfloor}{n-l-1 \brack l-1}(ab)^{\left\lfloor\frac{n}{2}\right\rfloor-l}q^{(l-1)^2}\right)\\
&=a^{\xi(n)}\left((ab)^{\left\lfloor\frac{n-1}{2}\right\rfloor+\xi(n+1)}+\sum_{l=1}^{\left\lfloor\frac{n-1}{2}\right\rfloor}{n-l-1\brack l}(ab)^{\left\lfloor\frac{n-1}{2}\right\rfloor-l+\xi(n+1)}q^{l^2} \right. \\&
\ \ \left. + q^{n-1}(1-\xi(n))q^{((n-2)/2)^2} +q^{n-1}\sum_{l=1}^{\left\lfloor\frac{n-1}{2}\right\rfloor}{n-l-1 \brack l-1}(ab)^{\left\lfloor\frac{n}{2}\right\rfloor-l}q^{(l-1)^2}\right).
\end{align*}
Note that $\left\lfloor\frac{n-1}{2}\right\rfloor+\xi(n+1)=\left\lfloor\frac{n}{2}\right\rfloor$  and from the well known recurrence of the $q$-binomial coefficient
\begin{align}
{n\brack j}=q^{n-j}{n-1\brack j-1}  + {n-1\brack j},
\end{align}
we get
\begin{align*}
F_{n+1}(q)&=a^{\xi(n)}\left((ab)^{\left\lfloor\frac{n}{2}\right\rfloor}+ q^{\frac{n^2}{4}}(1-\xi(n))  + \sum_{l=1}^{\left\lfloor\frac{n-1}{2}\right\rfloor}\left({n-l-1\brack l}+{n-l-1\brack l-1}q^{n-2l}\right)(ab)^{\left\lfloor\frac{n}{2}\right\rfloor-l}q^{l^2}\right) \\
&=a^{\xi(n)}\left((ab)^{\left\lfloor\frac{n}{2}\right\rfloor}+ q^{\frac{n^2}{4}}(1-\xi(n))  + \sum_{l=1}^{\left\lfloor\frac{n-1}{2}\right\rfloor}{n-l\brack l}(ab)^{\left\lfloor\frac{n}{2}\right\rfloor-l}q^{l^2}\right) \\
&=a^{\xi(n)}\sum_{l=0}^{\left\lfloor\frac{n}{2}\right\rfloor}{n-l\brack l}(ab)^{\left\lfloor\frac{n}{2}\right\rfloor-l}q^{l^2}.
\end{align*}
\end{proof}

 Note that if we consider the limit when $q$ tends to $1$ in above Theorem  we obtain the identity  (\ref{ec20}).

\begin{corollary}[Andrews, \cite{Andrews2}]
The following equality holds for any integer $n\geqslant 0:$
\begin{align*}
D_n(q)=\sum_{l=0}^{\left\lfloor\frac{n-1}{2}\right\rfloor}{n-l-1 \brack l}q^{l^2}.
\end{align*}
\end{corollary}

Now, we are going to prove a $q$-analogue of this Fibonacci identity:
$$F_{m-1}F_{m+k}-F_{m+k-1}F_m=(-1)^mF_k.$$

We follow the ideas of Andrews and et. al \cite{Andrews}.  We need the auxiliary sequence $\widehat{F}_{n}(q)$ defined by
\begin{align}
\widehat{F}_{0}(q)=1, \  \ \widehat{F}_{1}(q)=0,  \   \  \widehat{F}_{n}(q)=\begin{cases} a\widehat{F}_{n-1}(q) + q^{n-2}\widehat{F}_{n-2}(q),  \ \text{if} \ n\equiv 0 \pmod 2;\\
b\widehat{F}_{n-1}(q) + q^{n-2}\widehat{F}_{n-2}(q),  \ \text{if} \  n\equiv 1\pmod 2,
\end{cases} \label{eqq2}
\end{align}
where $n\geqslant 2$.  The first few terms are
$$1, \quad 0, \quad 1, \quad b, \quad a b + q^2, \quad a b^2 + b q^2 + b q^3,  \quad a^2 b^2 + a b q^2 + a b q^3 + a b q^4 + q^6, \dots$$

\begin{theorem}
The following equality holds for any integer $n\geqslant 1$
\begin{align*}
\widehat{F}_n(q)=b^{\xi(n)}\sum_{l=0}^{\left\lfloor\frac{n-2}{2}\right\rfloor}{n-l-2 \brack l}(ab)^{\left\lfloor\frac{n-2}{2}\right\rfloor-l}q^{l^2+l}.
\end{align*}
\end{theorem}
\begin{proof}
The proof runs like in  Theorem \ref{Theo1}.
\end{proof}
\begin{theorem} \label{theo4}
The following equality holds for integers $n\geqslant 0$ and $k\geqslant 1$,
\begin{multline}
F_n(q)\widehat{F}_{n+k}(q)-F_{n+k}(q)\widehat{F}_{n}(q)\\
=(-1)^{n-1}q^{\binom n2}\sum_{j=0}^{k-1}{k-1-j\brack j} a^{\left\lfloor\frac{k-1}{2}\right\rfloor + \xi(k+1)\xi(n+1)-j}b^{\left\lfloor\frac{k-1}{2}\right\rfloor + \xi(k+1)\xi(n)-j}q^{j^2+nj}. \label{theorog}
\end{multline}
\end{theorem}
\begin{proof}
Let $g_{n,k}$ be the right side of (\ref{theorog}) and  $f_{n,k}$ be the left side of (\ref{theorog}). Sequences $\{g_{n,k}\}$ and $\{f_{n,k}\}$ satisfy the same recurrence  and the same initial conditions.  We are going to prove that the sequence $\{g_{n,k}\}$ satisfies the  recurrence $g_{n,k}-a^{\xi(n+k+1)}b^{\xi(n+k)}g_{n,k-1}=q^{n-k+2}g_{n,k-2}$.

In fact,

\begin{align*}
g_{n,k}&-a^{\xi(n+k+1)}b^{\xi(n+k)}g_{n,k-1}\\
&=(-1)^nq^{\binom n2} \left( \sum_{j=0}^{k-1}\left({k-1-j\brack j}a^{\left\lfloor\frac{k-1}{2}\right\rfloor + \xi(k+1)\xi(n+1)-j}b^{\left\lfloor\frac{k-1}{2}\right\rfloor + \xi(k+1)\xi(n)-j}\right.\right.\\
& \ -\left.\left.{k-2-j\brack j} a^{\left\lfloor\frac{k-2}{2}\right\rfloor + \xi(k)\xi(n+1)+\xi(n+k+1)-j}b^{\left\lfloor\frac{k-2}{2}\right\rfloor + \xi(k)\xi(n)+\xi(n+k)-j} \right)q^{j^2+nj}\right) \\
&=(-1)^nq^{\binom n2} \left( \sum_{j=0}^{k-1}{k-1-j\brack j}a^{\left\lfloor\frac{k-2}{2}\right\rfloor + \xi(k)+\xi(k+1)\xi(n+1)-j}b^{\left\lfloor\frac{k-2}{2}\right\rfloor +\xi(k) +\xi(k+1)\xi(n)-j}\right.\\
& \  -\left.{k-2-j\brack j} a^{\left\lfloor\frac{k-2}{2}\right\rfloor + \xi(k)+\xi(k+1)\xi(n+1)-j}b^{\left\lfloor\frac{k-2}{2}\right\rfloor +\xi(k) +\xi(k+1)\xi(n)-j} \right)q^{j^2+nj} \\
&=(-1)^nq^{\binom n2} \left( \sum_{j=0}^{k-1}\left({k-1-j\brack j} - {k-2-j\brack j}\right)a^{\left\lfloor\frac{k-2}{2}\right\rfloor + \xi(k)+\xi(k+1)\xi(n+1)-j}b^{\left\lfloor\frac{k-2}{2}\right\rfloor +\xi(k) +\xi(k+1)\xi(n)-j} \right)q^{j^2+nj} \\
&=(-1)^nq^{\binom n2} \sum_{j=0}^{k-2}q^{k-1-2j}{k-2-j\brack j-1} a^{\left\lfloor\frac{k-2}{2}\right\rfloor + \xi(k)+\xi(k+1)\xi(n+1)-j}b^{\left\lfloor\frac{k-2}{2}\right\rfloor +\xi(k) +\xi(k+1)\xi(n)-j}q^{j^2+nj} \\
&=(-1)^nq^{\binom n2} \sum_{j=0}^{k-2}{k-2-j\brack j-1} a^{\left\lfloor\frac{k-2}{2}\right\rfloor + \xi(k)+\xi(k+1)\xi(n+1)-1-j}b^{\left\lfloor\frac{k-2}{2}\right\rfloor +\xi(k) +\xi(k+1)\xi(n)-1-j}q^{(j+1)^2+n+k-2} \\
&=(-1)^nq^{\binom n2 + n + k-2} \sum_{j=0}^{k-3}{k-3-j\brack j} a^{\left\lfloor\frac{k-3}{2}\right\rfloor + \xi(k-1)\xi(n+1)-j}b^{\left\lfloor\frac{k-3}{2}\right\rfloor  +\xi(k-1)\xi(n)-j}q^{j^2+nj} \\
&=q^{n+k-2}g_{n,k-2}.
\end{align*}

On the other hand,
\begin{align*}
f_{n,k}&=F_n(q)\widehat{F}_{n+k}(q)-F_{n+k}(q)\widehat{F}_{n}(q)\\
&=\begin{vmatrix} F_n(q) & F_{n+k}(q) \\
\widehat{F}_{n}(q) & \widehat{F}_{n+k}(q)
\end{vmatrix}\\
&=\begin{vmatrix} F_n(q) & a^{\xi(n+k+1)}b^{\xi(n+k)}F_{n+k-1}(q) + q^{n+k-2}F_{n+k-2} \\
\widehat{F}_{n}(q) & a^{\xi(n+k+1)}b^{\xi(n+k)}\widehat{F}_{n+k-1}(q) + q^{n+k-2}\widehat{F}_{n+k-2}
\end{vmatrix}\\
&=a^{\xi(n+k+1)}b^{\xi(n+k)}f_{n,k-1}+q^{n+k-2}f_{n,k-2}.
\end{align*}
Then the Equation (\ref{theorog}) follows.
\end{proof}
In particular, if $k=1$
\begin{align}
F_n(q)\widehat{F}_{n+1}(q)-F_{n+1}(q)\widehat{F}_{n}(q)=(-1)^{n-1}q^{\binom n2}.
\end{align}

If $a=b=1$ we obtain the  the following corollary.
\begin{corollary}[Andrews and et al. \cite{Andrews}]
The following equality holds for integers $n\geqslant 0$ and $k\geqslant 1$,
\begin{align}
F_n(q)\widehat{F}_{n+k}(q)-F_{n+k}(q)\widehat{F}_{n}(q)=(-1)^{n}q^{\binom n2}\sum_{j=0}^{k-1}{k-1-j\brack j} q^{j^2+nj}.
\end{align}
\end{corollary}

\section{Generating Function}
Edson and Yeyenie \cite{edson1} found the generating function to the biperiodic Fibonacci sequence:
\begin{align}
\F(x):=\sum_{n=0}^{\infty}t_nx^n=\frac{x(1+ax-x^2)}{1-(ab+2)x^2+x^4}. \label{genfun}
\end{align}
In this section we follow the ideas of Andrews \cite{Andrews2} to find a $q$-analogue of  (\ref{genfun}).  The Fibonacci operator $\eta$ was introduced by Andrews \cite{Andrews2}  by $\eta f(x)=f(qx)$. Moreover, we define the operator $\eta_2$ by $\eta_2f(x)=(qx)^2f(qx)$.

\begin{theorem}\label{theogen}
The generating function for the $q$-biperiodic Fibonacci sequence  is
\begin{align}
\W(x):=\sum_{n=0}^{\infty}F_n(q)x^n=\frac{1}{1-bx-x^2\eta}\left(x+(a-b)xf(x)\right),  \label{genfun1}
\end{align}
where $1/(1-bx-x^2\eta)$ is the inverse operator of $1-bx-x^2\eta$, and
\begin{align}
f(x):=\sum_{n=1}^{\infty}F_{2n-1}(q)x^{2n-1}=\frac{1}{1-abx^2-(1+1/q)x^2\eta + x^2\eta_2}\left(x-x^3\right). \label{genfun2}
\end{align}
\end{theorem}

\begin{proof}
We begin with the formal power series representation of the generating function for $\left\{F_n(q)\right\}_{n\in\N}$,
\begin{align*}
\W(x)=F_0(q)+F_1(q)x+F_2(q)x^2+\cdots + F_k(q)x^k + \cdots.\\
\end{align*}
Note that
\begin{align*}
bx\W(x)&=\sum_{n=1}^{\infty}bF_{n-1}(q)x^n,\\
x^2\eta\W(x)&=\sum_{n=2}^{\infty}F_{n-2}(q)q^{n-2}x^n.\end{align*}
Since $F_{2k+1}(q)=bF_{2k}+q^{2k-1}F_{2k-1}$, we get
\begin{align*}
(1-bx-x^2\eta)\W(x)=x+\sum_{n=1}^{\infty}(F_{2n}(q)-bF_{2n-1}(q)-q^{2n-2}F_{2n-2}(q))x^{2n}.
\end{align*}
Since $F_{2k}(q)=aF_{2k-1}+q^{2k-2}F_{2k-2}$, we get
\begin{align*}
(1-bx-x^2\eta)\W(x)&=x+(a-b)\sum_{n=1}^{\infty}F_{2n-1}(q)x^{2n}=x+(a-b)xf(x).
\end{align*}
Then Equation (\ref{genfun1}) is clear. \\
On the other hand,
\begin{align*}
F_{2n+1}(q)&=bF_{2n}(q)+q^{2n-1}F_{2n-1}(q)\\
&=b(aF_{2n-1}(q)+q^{2n-2}F_{2n-2}(q))+q^{2n-1}F_{2n-1}(q)\\
&=(ab+q^{2n-1})F_{2n-1}(q)+bq^{2n-2}F_{2n-2}(q)\\
&=(ab+q^{2n-1})F_{2n-1}(q)+q^{2n-2}(F_{2n-1}(q)-q^{2n-3}F_{2n-3}(q))\\
&=(ab+q^{2n-1}(1+1/q))F_{2n-1}(q)-q^{4n-5}F_{2n-3}(q).
\end{align*}
Then
\begin{align*}
(1-abx^2-(1+1/q)x^2\eta + x^2\eta_2)f(x)=x+F_3(q)x^3-abx^3-(1+1/q)qx^3=x-x^3.
\end{align*}
Therefore Equation (\ref{genfun2}) is obtained.
\end{proof}

If $a=b=1$ we obtain the following corollary
\begin{corollary}[Andrews, \cite{Andrews2}]
The generating function for the $q$-Fibonacci polynomials is
$$\sum_{n=0}^{\infty}D_n(q)x^n=\frac{x}{1-x-x^2\eta}.$$
\end{corollary}

\begin{lemma}\label{lemmag}
The following equality holds for any integer  $n\geqslant 0$.
\begin{align}
(bx+x^2\eta)^nx=x^{n+1}\sum_{j=0}^{n}b^{n-j}x^jq^{j^2}{n\brack j}.  \label{lemma1}
\end{align}
\end{lemma}
\begin{proof}
We proceed by induction on $n$. The equality clearly holds for $n=0$. Now suppose that the result is true for all $i\leqslant n$.
\begin{align*}
(bx+x^2\eta)^{n+1}x&=(bx+x^2\eta)(bx+x^2\eta)^{n}x=(bx+x^2\eta)x^{n+1}\sum_{j=0}^{n}b^{n-j}x^jq^{j^2}{n\brack j}\\
&=bx^{n+2}\sum_{j=0}^{n}b^{n-j}x^jq^{j^2}{n\brack j} +x^{n+3}\sum_{j=0}^{n}b^{n-j}x^jq^{j^2+j+n+1}{n\brack j}\\
&=bx^{n+2}\left(\sum_{j\geqslant0}b^{n-j}x^jq^{j^2}\left({n\brack j} +q^{n+1-j}{n\brack j-1}\right)\right)\\
&=x^{n+2}\left(\sum_{j\geqslant0}^{n}b^{n+1-j}x^jq^{j^2}{n+1\brack j}\right).
\end{align*}
Then the Equation (\ref{lemma1}) follows.
\end{proof}

\begin{theorem} The generating function for the $q$-biperiodic  Fibonacci sequence can be expressed
as
\begin{align}
\sum_{n=0}^{\infty}F_n(q)x^n=\sum_{j=0}^{\infty}\left(1+(a-b)f(x)\right)\frac{q^{j^2}x^{2j+1}}{(bx;q)_{j+1}}.
\end{align}
\end{theorem}
\begin{proof}
From Theorem \ref{theogen} and Lemma \ref{lemmag}  we have
\begin{align*}
\sum_{n=0}^{\infty}F_n(q)x^n&=\frac{1}{1-bx-x^2\eta}\left(x+(a-b)xf(x)\right)\\
&=\sum_{n=0}^{\infty}(bx+x^2\eta)^n(x+(a-b)xf(x))\\
&=\sum_{n=0}^{\infty}(bx+x^2\eta)^nx+(a-b)\sum_{n=0}^{\infty}(bx+x^2\eta)^nxf(x)\\
&=\sum_{n=0}^{\infty}x^{n+1}\sum_{j=0}^{n}b^{n-j}x^jq^{j^2}{n\brack j}+ (a-b)\sum_{n=0}^{\infty}x^{n+1}\sum_{j=0}^{n}b^{n-j}x^jq^{j^2}{n\brack j}f(x)\\
&=\sum_{n=0}^{\infty}\sum_{j=0}^{n}b^{n-j}x^{n+1+j}q^{j^2}{n\brack j}+ (a-b)\sum_{n=0}^{\infty}\sum_{j=0}^{n}b^{n-j}x^{n+1+j}q^{j^2}{n\brack j}f(x)\\
&=\sum_{n,j\geqslant 0}b^{n}x^{n+1+2j}q^{j^2}{n+j\brack j}+ (a-b)\sum_{n,j\geqslant 0}b^{n}x^{n+1+2j}q^{j^2}{n+j\brack j}f(x)
\end{align*}
From the $q$-analogue of the binomial series
$$\sum_{i=0}^{\infty}{n+i\brack i}x^i=\frac{1}{(x;q)_{n+1}},$$
we obtain
\begin{align*}
\sum_{n=0}^{\infty}F_n(q)x^n=\sum_{j =0}^{\infty}x^{2j+1}q^{j^2}\frac{1}{(bx;q)_{j+1}}+ (a-b)\sum_{j=0}^{\infty}x^{2j+1}q^{j^2}\frac{1}{(bx;q)_{j+1}}f(x).
\end{align*}
\end{proof}

\begin{corollary}[Andrews, \cite{Andrews2}]
The generating function for the $q$-Fibonacci polynomials is
\begin{align}
\sum_{n=0}^{\infty}D_n(q)x^n=\sum_{j=0}^{\infty}\frac{q^{j^2}x^{2j+1}}{(x;q)_{j+1}}.
\end{align}

\end{corollary}

\section{Combinatorial Interpretation}
The Fibonacci  numbers $F_{n+1}$  can be interpreted as the number of tilings of a board of length $n$ ($n$-board) with cells labelled 1 to $n$ from left to right with only squares and dominoes (cf. \cite{Benja}). This interpretation has been used to give a combinatorial interpretation of the $q$-Fibonacci polynomials and similar recurrent polynomials,  see, for instance \cite{Shatu, Briggs, Kris, Little, Stabel}.
In this section, we use tilings weighted by bicolored squares to give a combinatorial interpretation of the $q$-biperiodic Fibonacci sequence. We define a tiling weighted by bicolored squares as a tiling of a $n$-board by colored squares and non-colored dominoes, such that if the square has an odd position then there are $a$ different colors  to choose for the square. If the square has an even position then there are $b$ different colors to choose for the square. Moreover, if  a domino covering the $i$-th boundary receives weight $q^i$ (by $i$th boundary, we mean the boundary between cells $i$ and $i+1$, $1\leqslant i \leqslant n-1$).  Let $\T_n$ denote the set of all $n$-tilings, we shall show  Theorem \ref{theocount} that the biperiodic Fibonacci sequence $F_{n+1}(q)$ counts the number of tilings weighted by bicolored squares of a $n$-board.  Specifically, we have
$$F_{n+1}(q)=\sum_{T\in\T_n}a^{|T|_a}b^{|T|_b}q^{|T|},$$
where $|T|_a$ ($|T|_b$)  is the sum of all $i$ such that $T$ has a square in an odd position $i$ (even position $i$), and  $|T|$ is the sum of all $i$ such that $T$ has a domino in position $(i,i+1)$.\\

For example, in Figure \ref{tile} we show the different ways to tiling a $4$-board. Then it is clear that $F_5(q)=a^2b^2+abq+abq^2+abq^3+q^4.$

        \begin{figure}[h]
        \centering
   \includegraphics[scale=1]{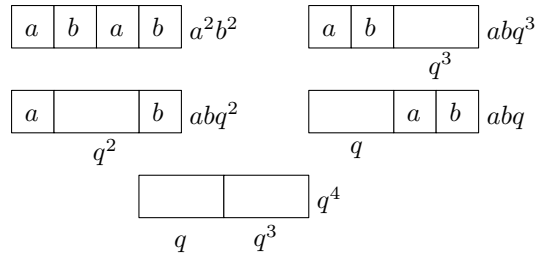}
      \caption{Different ways to tile $4$-boards.}
      \label{tile}
       \end{figure}

          \begin{theorem} \label{theocount}
For $n\geqslant 0$, $F_{n+1}(q)$ counts the number of tilings weighted by bicolored squares of a $n$-board.
\end{theorem}
\begin{proof}
Given $T\in\T_n$. If $n$ is even and $T$ ends with a domino, there  are $q^{n-1}F_{n-1}(q)$ ways to tile the board, and if $T$ ends with a square, there are $bF_n(q)$ ways to tile the $n$-board. Analogously, if $n$ is odd we get $F_{n+1}(q)=aF_{n}(q) + q^{n-1}F_{n-1}$. Moreover, it is clear that the initial values are 1 and $a$.
\end{proof}

 Let $g(n,k)$ be the number of tilings weighted by bicolored squares having $n$ tiles and $k$ dominoes. Then
 $$g(n,k)=q^{n+k-1}g(n-1,k-1) + a^{\xi(n+k)}b^{1-\xi(n+k)}g(n-1,k).$$

 In fact, if the $n+k$-board ends in a domino,  the domino contributes $q^{n+k-1}$ to the weight. Then, there are $q^{n+k-1}g(n-1,k-1)$ ways  to tile the board. If the last tile is a square there are $a^{\xi(n+k)}b^{1-\xi(n+k)}g(n-1,k)$ ways to tile the board. Let
 $$h(n,k)=a^{\xi(n+k)}q^{k^2}{n\brack k}(ab)^{\left\lfloor \frac{n-k}{2}\right\rfloor}.$$
The sequence $h(n,k)$ satisfies the same recurrence of $g(n,k)$. In fact,

\begin{align*}
q^{n+k-1}&h(n-1,k-1) + a^{\xi(n+k)}b^{1-\xi(n+k)}h(n-1,k)\\
&=q^{n+k-1}\left(a^{\xi(n+k)}q^{(k-1)^2}{n-1\brack k-1}(ab)^{\left\lfloor \frac{n-k}{2}\right\rfloor}\right) + a^{\xi(n+k)}b^{1-\xi(n+k)}a^{\xi(n+k-1)}q^{k^2}{n-1\brack k}(ab)^{\left\lfloor \frac{n-k-1}{2}\right\rfloor}\\
&=q^{k^2+n-k}a^{\xi(n+k)}{n-1\brack k-1}(ab)^{\left\lfloor \frac{n-k}{2}\right\rfloor} + ab^{1-\xi(n+k)}q^{k^2}{n-1\brack k}(ab)^{\left\lfloor \frac{n-k}{2}\right\rfloor - \xi(n+k+1)}\\
&=q^{k^2+n-k}a^{\xi(n+k)}{n-1\brack k-1}(ab)^{\left\lfloor \frac{n-k}{2}\right\rfloor} + a^{1-\xi(n+k+1)}q^{k^2}{n-1\brack k}(ab)^{\left\lfloor \frac{n-k}{2}\right\rfloor }\\
&=q^{k^2+n-k}a^{\xi(n+k)}{n-1\brack k-1}(ab)^{\left\lfloor \frac{n-k}{2}\right\rfloor} + a^{\xi(n+k)}q^{k^2}{n-1\brack k}(ab)^{\left\lfloor \frac{n-k}{2}\right\rfloor }\\
&=a^{\xi(n+k)}q^{k^2}\left(q^{n-k}{n-1\brack k-1}+{n-1\brack k}\right)(ab)^{\left\lfloor \frac{n-k}{2}\right\rfloor}\\
&=a^{\xi(n+k)}q^{k^2}{n\brack k}(ab)^{\left\lfloor \frac{n-k}{2}\right\rfloor}=h(n,k).
\end{align*}
Moreover, they satisfy the same initial conditions.  Therefore, we have the following lemma.

 \begin{lemma}\label{lema}
 The number of  tilings weighted by bicolored squares having $n$ tiles and $k$ dominoes is
 $$a^{\xi(n+k)}q^{k^2}{n\brack k}(ab)^{\left\lfloor \frac{n-k}{2}\right\rfloor}.$$
 \end{lemma}

Now, we will give a combinatorial proof of Theorem \ref{Theo1}, i.e.,
\begin{align*}
F_{n+1}(q)=a^{\xi(n)}\sum_{l=0}^{\left\lfloor\frac{n}{2}\right\rfloor}{n-l \brack l}(ab)^{\left\lfloor\frac{n}{2}\right\rfloor-l}q^{l^2}.
\end{align*}
\textit{Combinatorial Proof of Theorem \ref{Theo1}}.
 From Theorem \ref{theocount}, $F_{n+1}(q)$  counts the  number of  tilings weighted by bicolored squares of a $n$-board.  On the other hand, let $l$ be the number of dominoes in the tiling of a $n$-board. Then there are $n-2l$ squares. Such a tiling with $n-l$ tiles, exactly $l$ of which are dominoes is
$$a^{\xi(n)}q^{l^2}{n-l\brack l}(ab)^{\left\lfloor \frac{n}{2}\right\rfloor-l}.$$
Summing over all possible $l$ gives the identity. \qedhere

\subsection{Additional Identities}
In this section, we prove several $q$-analogues of Fibonacci identities using tilings weighted by bicolored squares.

\begin{theorem}
The following equality holds for any integer $n\geqslant 0:$
\begin{align}
\sum_{k=1}^{n+1}q^kF_k(q)a^{\xi(n)\xi(k)}b^{\xi(k+1)-\xi(n)\xi(k+1)}(ab)^{\left\lfloor\frac{n-k+1}{2}\right\rfloor}=F_{n+3}(q)-a^{\xi(n)}(ab)^{\left\lfloor\frac{n+2}{2}\right\rfloor}. \label{id1}
\end{align}
\end{theorem}
\begin{proof}
There exists $F_{n+3}(q)-a^{\xi(n)}(ab)^{\left\lfloor\frac{n+2}{2}\right\rfloor}$ tilings weighted by bicolored squares of a $n+2$-board with at least one domino. On the other hand, consider the location of the last domino, say position $(k,k+1)$. This domino contributes a $q^k$ to the weight, all tiles to the right are squares and contribute (it depends of the parity of the numbers $k$ and $n$) an $a^{\xi(n)\xi(k)}b^{\xi(k+1)(1-\xi(n))}(ab)^{\left\lfloor\frac{n-k+1}{2}\right\rfloor}$ to the weight.  Moreover, there are $F_k(q)$ ways to tile the left side ($k-1$-board). Summing over all possible $k$ gives the identity.
\end{proof}

 Note that if we consider the limit when $q$ tends to $1$ in above theorem  we obtain the new identity
 $$\sum_{k=1}^{n+1}t_ka^{\xi(n)\xi(k)}b^{\xi(k+1)-\xi(n)\xi(k+1)}(ab)^{\left\lfloor\frac{n-k+1}{2}\right\rfloor}=t_{n+3}-a^{\xi(n)}(ab)^{\left\lfloor\frac{n+2}{2}\right\rfloor}.$$

Equation (\ref{id1}) is a $q$-analogue of the following Fibonacci identity (\cite{koshy}):
$$
\sum_{k=0}^{n}F_k=F_{n+2}-1.
$$

The following theorem is a $q$-analogue of the Fibonacci identity (\cite{koshy}):
$$\sum_{k=0}^{n}F_{2k}=F_{2n+1}.$$

\begin{theorem} The following equality holds for any integer $n\geqslant 0:$
\begin{align}
a\sum_{k=0}^{n}F_{2k+1}(q)q^{n^2+n-(k^2+k)}=F_{2n+2}(q).
\end{align}
\end{theorem}
\begin{proof}
There exists $F_{2n+2}(q)$ tilings weighted by bicolored squares of a $2n+1$-board. On the other hand, consider the location of the last square, say position $k$.  Since the length of the board is odd then $k$ is odd and it contributes an $a$ to the weight. The dominoes to the right contributes a $q^{(k+1)+(k+3)+\cdots+(2n)}=q^{n^2+n-(k^2-1)/4}$ to the weight.  Moreover, there are $F_k(q)$ ways to tile the left side ($k-1$-board). Summing over all possible odd $k$ gives the identity.
\end{proof}

If we consider the limit when $q$ tends to $1$ in above theorem  we obtain the identity
 $$a\sum_{k=0}^{n}t_{2k+1}=t_{2n+2}.$$

 We need the following shifted $q$-biperiodic Fibonacci sequence.

\begin{align*}
&F^{(s)}_{0}(q)=0, \quad F^{(s)}_{1}(q)=1,  \quad \\
&F^{(s)}_{n}(q)=\begin{cases} a^{\xi(s+1)}b^{1-\xi(s+1)}F^{(s)}_{n-1}(q) + q^{n-2+s}F^{(s)}_{n-2}(q),  \ \text{if} \ n\equiv 0 \pmod 2;\\
a^{1-\xi(s+1)}b^{\xi(s+1)}F^{(s)}_{n-1}(q) + q^{n-2+s}F^{(s)}_{n-2}(q),  \ \text{if}  \  n\equiv 1\pmod 2;
\end{cases} \ n\geqslant 2. \label{eqq1}
\end{align*}

A tiling of a $n$-board is breakable at cell $k$, if the tiling can be decomposed into two tilings, one covering cells 1 through $k$ and the other covering cells $k+1$ through $n$. Moreover, a tiling of a $n$-board is unbreakable at cell $k$ if a domino occupies cell $k$ and $k+1$ (cf. \cite{Benja}).

It is not difficult to show that $F^{(s)}_{n+1}(q)$ counts the latter position of the tilings weighted by bicolored squares of a $n+s-1$-board that can be breakable at cell $s$.

\begin{theorem}
The following equality holds for any integers $n, m \geqslant 0:$
\begin{align}
F_{n+m+1}(q)=F_{m+1}(q)F_{n+1}^{(m)}(q) + q^mF_{m}(q)F_{n}^{(m+1)}(q).
\end{align}
\end{theorem}
\begin{proof}
There exists $F_{n+m+1}(q)$ tilings weighted by bicolored squares of a $n+m$-board. On the other hand, we will consider two cases. If a $n+m$-tiling is breakable at cell $m$, we have $F_{m+1}(q)$ ways to tile a $m$-board (left side) and $F_{n+1}^{(m)}(q)$ ways to tile the $n$-board (right side).   If a $n+m$-tiling is unbreakable at cell $m$   then there is a domino in position $(m, m+1)$. It  contributes  a $q^{m}$ to the weight.  Moreover, there are  $F_{m}(q)$ ways to tile a $m-1$-board (left side) and $F_{n}^{(m+1)}(q)$ ways to tile the $n-1$-board (right side).
\end{proof}
The above theorem is a $q$-analogue of the Fibonacci identity (\cite{koshy}):
$$F_{m+n}=F_mF_n+F_{m-1}F_{n-1}.$$

The following theorem is a $q$-analogue of the biperiodic Fibonacci identity (see Theorem 7 of \cite{edson1}):
$$\sum_{k=0}^n\binom nk a^{\xi(k)}(ab)^{\left\lfloor \frac k2 \right\rfloor}t_k=t_{2n}.$$

\begin{theorem}
The following equality holds for any integer $n\geqslant 0:$
\begin{align}
F_{2n}(q)=\sum_{k=1}^na^{\xi(k)}q^{(n-k)^2} {n\brack k}(ab)^{\left\lfloor \frac k2 \right\rfloor} F_{k}^{(2n-k)}(q).
\end{align}
\end{theorem}
\begin{proof}
There exists $F_{2n}(q)$ tilings weighted by bicolored squares of a $2n-1$-board. On the other hand, note that a $2n-1$-tiling have to include at least $n$ tiles, and one of them if a square. If  a $2n-1$-tiling have $k$ squares and $n-k$ dominoes  among the first $n$ tiles, then  by Lemma  \ref{lema}  there are $$a^{\xi(2n-k)}q^{(n-k)^2}{n\brack n-k}(ab)^{\left\lfloor \frac{k}{2}\right\rfloor}=a^{\xi(k)}q^{(n-k)^2}{n\brack k}(ab)^{\left\lfloor \frac{k}{2}\right\rfloor}$$
ways to tile this board. The remainder right board has length $k-1$ and can be tile $F_{k}^{(2n-k)}(q)$ ways.
\end{proof}

\begin{theorem}
The following equality holds for any integer $n\geqslant 0:$
\begin{align}
F_{2n+2}(q)=a\sum_{i=0}^n\sum_{j=0}^n(ab)^{\xi(n-i-j)}q^{i^2+(n+i+j)j} {n-j\brack i}{n-i\brack j}(ab)^{2\left\lfloor \frac{n-i-j}{2} \right\rfloor}.
\end{align}
\end{theorem}
\begin{proof}
There exists $F_{2n+2}(q)$ tilings weighted by bicolored squares of a $2n+1$-board. On the other hand, note that a $2n+1$-tiling have an odd number of squares, then there is a square such that the number of squares to the left side of it is equal to the number of squares of the right side. This square is called median square and  it  contributes  an $a$ to the weight.  We will count the number of tilings contain exactly $i$ dominoes to the left of the median square and exactly $j$ dominoes to the right of the median square.  A tiling of a $2n+1$-board with $i+j$ dominoes  have $2n+1-2i-2j$ squares, then there are $n-i-j$ squares on each side of the median square. Since the left side has $n-j$  tiles $i$ of which are dominoes, then there are
 $$a^{\xi(n-j+i)}q^{i^2}{n-j\brack i}(ab)^{\left\lfloor \frac{n-i-j}{2}\right\rfloor}$$
 ways  to tile this board.   Analogously, the right side can be tile by
   $$b^{\xi(n-j+i)}q^{(n+i+1)j}{n-i\brack j}(ab)^{\left\lfloor \frac{n-i-j}{2}\right\rfloor}$$
   ways.
\end{proof}
The above theorem is a $q$-analogue of the Fibonacci identity (\cite{koshy})
$$\sum_{i=0}^n\sum_{j=0}^n\binom{n-i}{j}\binom{n-j}{i}=F_{2n+1}.$$

\section{Open Questions}

We lead the following open questions about $q$-biperiodic Fibonacci sequence.

Schur \cite{Schur} proved that
\begin{align*}
D_\infty=\prod_{n\geqslant 0}\frac{1}{(1-q^{5n+1})(1-q^{5n+4})},
\end{align*}
where $D_\infty=\lim_{n\to \infty}D_{n}(q)$. Is there any analogue to $F_{\infty}(q)$, where $F_\infty=\lim_{n\to \infty}F_{n}(q)$?

 The following is one of the Rogers-Ramanujan identities:

\begin{multline*}
\sum_{n\geqslant0}\frac{q^{n^2+mn}}{(1-q)(1-q^2)\cdots (1-q^n)}=(-1)^mq^{-\binom n2}E_{m-2}(q)\prod_{n\geqslant 0}\frac{1}{(1-q^{5n+1})(1-q^{5n+4})}\\ - (-1)^mq^{-\binom n2}D_{m-2}(q)\prod_{n\geqslant 0}\frac{1}{(1-q^{5n+2})(1-q^{5n+3})},
\end{multline*}
where $E_n(q)$ is the sequence defined as $E_n(q)=E_{n-1}(q)+ q^{n-2}E_{n-2}(q)$, with the initial conditions  $E_0(q)=1, E_1(q)=0$. Is there a similar identity that involves the sequences $F_n(q)$ and $\widehat{F}_{n}(q)$?\\

Andrews \cite{Andrews4} proved that
\begin{align*}
D_{n+1}(q)= \sum_{j=-\infty}^{\infty}(-1)^jq^{j(5j+1)/2}{n\brack \left\lfloor \frac{n-5j}{2} \right\rfloor}.
\end{align*}
Is there any analogue  identity to the sequence $F_n(q)$?

\section{Acknowledgements}

The first author was partially supported by Universidad Sergio Arboleda under Grant no. DII-262.
The second author thanks  the invitation to Bogot\'a where the mayor part of this paper was done.

\end{document}